\documentclass[3p]{elsarticle}

\usepackage{rotate}          
\usepackage{url}

\usepackage{mathrsfs}
\usepackage{amsfonts}
\usepackage{amsmath}
\usepackage{amsthm}
\usepackage[usenames]{color}
\usepackage{footnote}
\usepackage{longtable}
\usepackage{hyperref}
\usepackage{float}
\usepackage{graphicx}
\usepackage{amssymb}
\usepackage{multicol}
\usepackage{subfig}
\usepackage{caption}

\newtheorem{Theorem}{Theorem}[section]
\newtheorem{Lemma}[Theorem]{Lemma}
\newtheorem{Corollary}{Corollary}[section]

\newtheorem{Definition}[Theorem]{Definition}

\journal{XXX}

\begin{document}

\begin{frontmatter}

\title{On Symmetric Invertible Binary Pairing Functions}

\author[]{Jianrui Xie\corref{mycorrespondingauthor}}
\cortext[mycorrespondingauthor]{Corresponding author}
\ead{jrxie93@stu.xidian.edu.cn}

\address[mymainaddress]{ISN Laboratory, Xidian University, Xi'an 710071, China}

\begin{abstract}
We construct a symmetric invertible binary pairing function $F(m,n)$ on the set of positive integers with a property of $F(m,n)=F(n,m)$. Then we provide a complete proof of its symmetry and bijectivity, from which the construction of symmetric invertible binary pairing functions on any custom set of integers could be seen.
\end{abstract}


\end{frontmatter}

\section{Introduction}

Pairing functions were first used to demonstrate that the cardinalities of the set of rationals (denoted by $\mathbb Q$) and the set of natural numbers (denoted by $\mathbb N^+$) are the same \cite{Pairingfunction}. They arise in coding problems as well, where a vector of integer values is to be folded onto a single integer value reversibly. For example, once we need to track pairs of integer values but the protocol, schema or API (Application Program Interface) only accept scalars, using pairing functions should be reserved as a hack of last resort if the system can not be modified to accommodate a collection. When referred to memory management, a pairing function could be a perfect hashing function essential to control the memory footprint but also to speed up computation\cite{Baisalov2001Fragments}.
They are even involved in engineering projects such as authentication of users \cite{Spalka2013Computer} and automating grammar comparison \cite{Madhavan2015Automating}.

Generally a pairing function $P(m,n)$ on a set $D$ is a bijection from $D \times D$ to $D$. It is invertible and associates each pair of members from $D$ with a single member of $D$ uniquely. The Cantor pairing function, which is on the set of natural numbers (denoted by $\mathbb {N}$), is shown as (1) below \cite{Cantor1900Ein}.
\begin{equation}
C(m,n)=\frac{1}{2} (m+n)(m+n+1)+n
\end{equation}
Let $m$ and $n$ represent the row and column indexes, the following pattern in Figure 1 offers an enumeration of $C(m,n)$. Various kind of pairing functions have been designed to fit increasing needs after George Cantor \cite{Lisi2007Some}\cite{Cegielski2001Decidability}\cite{Cegielski2000The}. So there is reason to believe in the potential values of symmetric binary pairing functions in the future.
\begin{gather*}
\bordermatrix
{&0&1&2&3&4&5&6&7&8&9& \cdots \cr
0&0 & 2 & 5 & 9 & 14 & 20 & 27&35&44&54& \cdots \cr
1&1  & 4 & 8 & 13 & 19 & 26&34&43&53&64& \cdots \cr
2&3  & 7 &12& 18 & 25&33&42&52&63&75&\cdots \cr
3&6  &11&17& 24&32&41&51&62&74&87&\cdots \cr
4&10&16&23&31&40&50&61&73&86&100&\cdots \cr
5&15&22&30&39&49&60&72&85&99&104&\cdots \cr
6&21&29&38&48&59&71&84&98&113&129&\cdots \cr
7&28&37&47&58&70&83&97&112&128&145&\cdots \cr
8&36&46&57&69&82&96&111&127&144&162&\cdots \cr
9&45&56&68&81&95&110&126&143&161&180&\cdots \cr
\vdots&\vdots&\vdots&\vdots&\vdots&\vdots&\vdots&\vdots&\vdots&\vdots&\vdots&\ddots \cr
}
\end{gather*}

\begin{center}
Figure 1: An enumeration of Cantor pairing function $C(m,n)$.
\end{center}

The present paper primarily focuses on a complete demonstration of the symmetry and bijectivity of the pairing function we are going to provide. Then if we reverse steps of the proof, a method to construct a corresponding family of symmetric invertible binary pairing functions on a family of integer set would be evident.

\section{Preliminaries}\label{sec:prel}
\noindent
For convenience of the proof we make declarations of a few necessary notations in a concise manner.

The $i$th element of an ordered set $A$ is denoted by $A(i)$. Two extremum functions $max(x,y)$ and $min(x,y)$ are defined as in probability theory.
\begin{equation}\label{}
max(x,y)=\frac {1}{2}(x+y+| x-y |)
\end{equation}
\begin{equation}\label{}
min(x,y)=\frac {1}{2}(x+y-| x-y |)
\end{equation}

On the basis of the signum function, we define that
\begin{eqnarray}
sgn^{\prime}(x)=\frac{sgn^2(x)+sgn(x)}{2}=\left\{\begin {array}{ll}
       0, & \textrm {if  $ x\le0$ }\\
       1,  &\textrm {if  $ x>0.$  }\\
       \end {array}\right.
\end{eqnarray}

Let $m \in \mathbb {N}$, we declare that for any given integer $n$ the least non-negative residue of $n \ modulo \ m$ is denoted by $n \ \% \ m$.

For $F(m,n)$, we present certain definitions of set theory as follows \cite{Foreman2010Handbook} \cite{Johnsonbaugh1986Discrete}.


\begin{Definition}\label{definition1}
$A$ is called to be an ordered subset of $B$ if\\
\indent (1) $A$ is a subset of $B$;\\
\indent (2) $A$ is ordered, $B$ is unordered.
\end{Definition}

\begin{Definition}\label{definition2}
$S_{U}$ is an unordered set of all positive integer pairs such that
\begin{equation}\label{}
S_{U}=\{(x,y) |x \in \mathbb {N^+},y \in \mathbb {N^+}\}.
\end{equation}
\end{Definition}

\begin{Definition}\label{definition3}
Let $i \in \mathbb {N^+} \setminus \{1\}$, $S_i$ is an ordered subset of $S_{U}$ if it satisfies:\\
\indent (1) For any $(x,y)\in S_i$, we have $x+y=i$ and $x\ge y$;\\
\indent (2) Let $(x_1,y_1)=S_i(a)\in S_i$, and $(x_2,y_2)=S_i(b)\in S_i$.   If $y_1<y_2$, then we have $a<b$, else if $y_1>y_2$, we have $a>b$.
\end{Definition}
\indent From DEFINITION \ref{definition3} we conclude that the size of $S_i$ is $ \lfloor {\frac {i}{2}} \rfloor $.

\begin{Definition}\label{definition4}
$S_{O}$ is an ordered subset of $S_{U}$ if it satisfies:\\
\indent (1) For any $i\in \{i | i \in \mathbb {N^+}, i \ge 2\}$, $S_i$ is an ordered subset of $S_{O}$;\\
\indent (2) For $(x,y) \in S_{O}$, $(x,y) \in S_{x+y}$;\\
\indent (3) For $(x_1,y_1) \in S_a$ and $(x_2,y_2) \in S_a$, if $(x_1,y_1)=S_{O}(m)$, $(x_2,y_2)=S_{O}(n)$ and $y_1<y_2$, then $m<n$;\\
\indent (4) For $(x_1,y_1) \in S_a$ and $(x_2,y_2) \in S_b$, if $(x_1,y_1)=S_{O}(m)$, $(x_2,y_2)=S_{O}(n)$ and $a<b$, then $m<n$.
\end{Definition}
\indent From DEFINITION \ref{definition4} we conclude that the size of $S_{O}$ is infinite.

\begin{Definition}\label{definition5}
A binary function $F(x,y)$ whose domain $D$ is an ordered subset of $S_{U}$ is called binary strictly monotonous if it satisfies either one of\\
\indent (1) For $(x_1,y_1) \in D$ and $(x_2,y_2) \in D$, if $(x_1,y_1)=D(i)$, $(x_2,y_2)=D(j)$, and $i<j$ then $F(x_1,y_1)<F(x_2,y_2)$ holds;\\
\indent (2) For $(x_1,y_1) \in D$ and $(x_2,y_2) \in D$, if $(x_1,y_1)=D(i)$, $(x_2,y_2)=D(j)$, and $i<j$ then $F(x_1,y_1)>F(x_2,y_2)$ holds.
\end{Definition}
\begin{Definition}\label{definition6}
A binary set $N$ is called to be the symmetric set of a binary set $M$\ (denoted by $S$-$M$)\ if\\
\indent (1) For $(x,y) \in M$, $(y,x) \in N$;\\
\indent (2) For $(x,y) \in N$, $(y,x) \in M$.
\end{Definition}


\section{Demonstration}\label{sec:Dem}
We now present the pairing function below that will be discussed mainly about.
\begin{equation}\label{}
F(m,n)=\frac{1}{4} \times [(m+n-1)^2-(m+n-1)\ \% \ 2]+min(m,n)
\end{equation}
It is an evident fact that $F(m,n)$ has a property of symmetry such that
$$\forall m\in \mathbb {N^+},\ n \in \mathbb {N^+},\ F(m,n)=F(n,m).$$


\begin{Lemma}\label{lemma1}
For any $ x \in \mathbb {N^+}$, both
\begin{equation}\label{}
x=\lfloor {\frac {x+1}{2}} \rfloor \times 2 - x \ \% \ 2
\end{equation}
and
\begin{equation}
x=\lfloor {\frac {x}{2}} \rfloor \times 2 + x \ \% \ 2
\end{equation}
always hold.
\end{Lemma}

\begin{Lemma}\label{lemma2}
Let $ (x_1,y_1) \in S_i$,\ $(x_2,y_2) \in S_i$.  Let $(x_1,y_1)=S_i(a)$ and $(x_2,y_2)=S_i(b)$. If $a<b$, then
\begin{equation*}
    F(x_1,y_1)<F(x_2,y_2).
\end{equation*}
\end{Lemma}
\begin{proof}
By DEFINITION \ref{definition3}, $x_1 \ge y_1$, $x_2 \ge y_2$. So $y_1<y_2$ if $a<b$.
Since $x_1+y_1=x_2+y_2$, we have $x_1>x_2$.
Therefore,
$$F(x_2,y_2)-F(x_1,y_1)=y_2-y_1>0.$$
\end{proof}

\begin{Lemma}\label{lemma3}
Let $ (x_1,y_1) \in S_i$, $(x_2,y_2) \in S_{i+1}$. We have
\[F(x_1,y_1)<F(x_2,y_2).\]
\end{Lemma}
\begin{proof}
By LEMMA \ref{lemma2}, it follows that for any $i \in \mathbb {N^+}\setminus\{ 1 \}$, we have
$$F(i-1,1)<F(i-2,2)<\cdots<F(i-\lfloor {\frac {i}{2}} \rfloor ,\lfloor {\frac {i}{2}} \rfloor),$$
where $F(m,n)$ is defined on $S_i$.
Similarly,
$$F(i,1)<F(i-1,2)<\cdots<F(i+1-\lfloor {\frac {i+1}{2}} \rfloor ,\lfloor {\frac {i+1}{2}} \rfloor),$$
where $F(m,n)$ is defined on $S_{i+1}$.
By LEMMA \ref{lemma1} we have
\begin{equation}\label{}
  F(i,1)-F(i-\lfloor {\frac {i}{2}} \rfloor ,\lfloor {\frac {i}{2}} \rfloor)=1,
\end{equation}
which implies $F(x_1,y_1)<F(x_2,y_2)$ for any $ (x_1,y_1) \in S_i$ and $(x_2,y_2) \in S_{i+1}$.
\end{proof}

\begin{Lemma}\label{lemma4}
For any $(x_1,y_1)\in S_i$ and $(x_2,y_2) \in S_j$, where $i<j$, we have
$$F(x_1,y_1)<F(x_2,y_2).$$
\end{Lemma}
\begin{proof}
From LEMMA \ref{lemma2} and LEMMA \ref{lemma3}, it follows that
\[F(x_1,y_1) \le F(i-\lfloor {\frac {i}{2}} \rfloor ,\lfloor {\frac {i}{2}} \rfloor) <F(j-1,1) \le F(x_2,y_2).\]
\end{proof}

\begin{Lemma}\label{lemma5}
Let $ (x_1,y_1)\in S_{O}$,\ $(x_2,y_2)\in S_{O}$.  Let $(x_1,y_1)=S_{O}(a)$ and $(x_2,y_2)=S_{O}(b)$. If $a<b$, then
\[F(x_1,y_1)<F(x_2,y_2).\]
\end{Lemma}
\begin{proof}
It could be easily deduced from LEMMA \ref{lemma2} and LEMMA \ref{lemma4}.
\end{proof}


\begin{Theorem}\label{theorem1}
$F(m,n)$ is an injection from $S_{O}$ to $\mathbb {N^+}$.
\end{Theorem}
\begin{proof}
According to DEFINITION \ref{definition5}, LEMMA \ref{lemma5} shows that $F(m,n)$ is binary strictly montonous on domain of $S_{O}$. Now for any different $(x_1,y_1)$ and $(x_2,y_2)$ in $S_{O}$, suppose that
$(x_1,y_1)=S_{O}(a),\ (x_2,y_2)=S_{O}(b)$. Then clearly $a \ne b$.
By LEMMA \ref{lemma5}, we have
$$F(x_1,y_1) \ne F(x_2,y_2)$$
and the proof is complete.
\end{proof}

\begin{Theorem}\label{theorem2}
$F(m,n)$ is a surjection from $S_{O}$ to $\mathbb {N^+}$.
\end{Theorem}
\begin{proof}
For any $i \in \mathbb {N^+} \setminus \{ 1 \}$, by LEMMA \ref{lemma2}, it could be deduced that on each $S_i$,
\begin{align*}
1&=F(i-2,2)-F(i-1,1)\\
&=F(i-3,3)-F(i-2,2)\\
&=\cdots \\
&=F(i-\lfloor {\frac {i}{2}} \rfloor ,\lfloor {\frac {i}{2}} \rfloor)-F(i-\lfloor {\frac {i}{2}} \rfloor+1 ,\lfloor {\frac {i}{2}} \rfloor-1).
\end{align*}
Then due to the proof of LEMMA \ref{lemma3},
$$F(i,1)-F(i-\lfloor {\frac {i}{2}} \rfloor ,\lfloor {\frac {i}{2}} \rfloor)=1.$$
As a result of induction, for any $(x_1,y_1)\in S_{O}$ and $(x_2,y_2)\in S_{O}$ with $(x_1,y_1)=S_{O}(m)$ and $(x_2,y_2)=S_{O}(n)$, we have
$$F(x_1,y_1)-F(x_2,y_2)=m-n.$$
Since $(1,1)=S_{O}(1)$ and $f(1,1)=1$,
it could be inferred that
\begin{equation}
F(x,y)=k,~\forall (x,y)=S_{O}(k).
\end{equation}
Hence, the infinity of $S_{O}$ implies that $F(m,n)$ on $S_{O}$ traverses $\mathbb {N^+}$, as asserted.
\end{proof}

Summarizing, we have proved the following.
\begin{Theorem}\label{theorem3}
$F(m,n)$ is a bijection from $S_{O}$ to $\mathbb {N^+}$.
\end{Theorem}

\indent We now generalize $F(m,n)$ to the set of positive integer pairs.
\begin{Theorem}\label{theorem4}
For any $(x_1,y_1) \ne (x_2,y_2) \in S_{U}$,\ $F(m,n)$ on $S_{U}$ has two properties:\\
(1) $F(x_1,y_1)=F(x_2,y_2)$,\ if $(x_1,y_1)=(y_2,x_2);$\\
(2) $F(x_1,y_1)\ne F(x_2,y_2)$,\ if $(x_1,y_1)\ne (y_2,x_2).$
\end{Theorem}
\begin{proof}
By THEOREM \ref{theorem3} and the symmetry of $F(m,n)$, it can be seen that $F(m,n)$ is also a bijection from $S$-$S_{O}$ to $\mathbb {N^+}$. Since$$S_{U}=S_{O}\cup S \textrm{-} S_{O},$$
$F(m,n)$ is a surjection from $S_{U}$ to $\mathbb {N^+}$. Therefore, the first property follows immediately as
$$F(x_1,y_1)=F(y_2,x_2)=F(x_2,y_2).$$
Noticing that
$$F(x_1,y_1)=F(max(x_1,y_1),min(x_1,y_1)),$$
$$F(x_2,y_1)=F(max(x_2,y_2),min(x_2,y_2)),$$
$$(max(x_1,y_1),min(x_1,y_1)) \in S_{O},$$
$$(max(x_2,y_2),min(x_2,y_2)) \in S_{O},$$
and
$$(max(x_1,y_1),min(x_1,y_1))\ne (max(x_2,y_2),min(x_2,y_2)),$$
the second property is true by THEOREM $1$.
\end{proof}
Thanks to the preceding theorems, we illustrate a partial enumeration of $F(m,n)$ on $S_{U}$ as Figure 2 if $(m,n)$ are row-column indexing:
\begin{gather*}
\bordermatrix
{&1&2&3&4&5&6&7&8&9&10& \cdots \cr
1&1 &2&3&5&7&10&13&17&21&26& \cdots \cr
2&2  & 4&6&8&11&14&18&22&27&32& \cdots \cr
3&3  & 6 &9&12&15&19&23&28&33&39&\cdots \cr
4&5 &8&12&16&20&24&29&34&40&46&\cdots \cr
5&7&11&15&20&25&30&35&41&47&54&\cdots \cr
6&10&14&19&24&30&36&42&48&55&62&\cdots \cr
7&13&18&23&29&35&42&49&56&63&71&\cdots \cr
8&17&22&28&34&41&48&56&64&72&80&\cdots \cr
9&21&27&33&40&47&55&63&72&81&90&\cdots \cr
10&26&32&39&46&54&62&71&80&90&100&\cdots \cr
\vdots&\vdots&\vdots&\vdots&\vdots&\vdots&\vdots&\vdots&\vdots&\vdots&\vdots&\ddots \cr
}
\end{gather*}

\begin{center}
Figure 2: An enumeration of $F(m,n)$.
\end{center}


\begin{Theorem}\label{theorem5}
For any $C \in \mathbb {N^+}$, $F(m,n)=C$ has the only solution for $m\geq n$.
Let$$t= \lfloor {\sqrt{C}} \rfloor,$$
$$a=sgn(C-t^2),$$
and
$$b=sgn^{\prime}(C-t^2-t).$$
Then the only solution is
\begin{eqnarray}
\left\{\begin {array}{ll}
       m=(1-a)t+a[(t+1)(t+1+b)-C]\\
       n=(1-a)t+a[C-t(t+b)].\\
       \end {array}\right.
\end{eqnarray}
\end{Theorem}
\begin{proof}
$$m-n=a(2t^2+2t+1+2bt+b-2C) \ge 0,$$
which implies  $m\ge n.$
Consider the following three cases:\\
(1) $a=0,\ b=0,\ if \ C=t^2$;\\
(2) $a=1,\ b=0,\ if \ t^2 \le C \ge t^2+t$;\\
(3) $a=1,\ b=1,\ if \ t^2+t<C<(t+1)^2$.\\
It can be proved that $$F(m,n)=C$$
always holds.
By THEOREM \ref{theorem3},\ we know that $(m,n)$, $m \ge n$,  is unique.
Thus, for any $C \in N^+$,
\begin{eqnarray}
\left\{\begin {array}{ll}
       m=(1-a)t+a[(t+1)(t+1+b)-C]\\
       n=(1-a)t+a[C-t(t+b)]\\
       \end {array}\right.
\end{eqnarray}
is the only solution to $$F(m,n)=C \ (m\ge n).$$
\end{proof}


\begin{Corollary}
\begin{equation}\label{}
    G(m,n)=\frac{1}{4}[(m+n+1)^2-(m+n+1)\ \% \ 2]+min(m,n)
\end{equation}
is a symmetric pairing function from $\mathbb {N}\times \mathbb {N}$ to $\mathbb {N}$ .
\end{Corollary}
Similarly, $G(m,n)$ follows with a parallel pattern of enumeration as what can be told from Figure 3. Note that both the row and column indexes start not from $1$ but from $0$.
\begin{gather*}
\bordermatrix{&0&1&2&3&4&5&6&7&8&9& \cdots \cr
               0&0 &1&2&4&6&9&12&16&20&25& \cdots \cr
               1&1  & 3&5&7&10&13&17&21&26&31& \cdots \cr
               2&2  & 5 &8&11&14&18&22&27&32&38&\cdots \cr
               3&4 &7&11&15&19&23&28&33&39&45&\cdots \cr
4&6&10&14&19&24&29&34&40&46&53&\cdots \cr
5&9&13&18&23&29&35&41&47&54&61&\cdots \cr
6&12&17&22&28&34&41&48&55&62&70&\cdots \cr
7&16&21&27&33&40&47&55&63&71&79&\cdots \cr
8&20&26&32&39&46&54&62&71&80&89&\cdots \cr
9&25&31&38&45&53&61&70&79&89&99&\cdots \cr
\vdots&\vdots&\vdots&\vdots&\vdots&\vdots&\vdots&\vdots&\vdots&\vdots&\vdots&\ddots \cr
            }
\end{gather*}

\begin{center}
Figure 3: An enumeration of $G(m,n)$.
\end{center}


\begin{Corollary}
Given any $C \in \mathbb {N}$, there is only one solution to $$G(m,n)=C \ (m\in \mathbb {N},\ n\in \mathbb {N},\ m\ge n).$$
Let$$t= \lfloor {\sqrt{C}} \rfloor,$$
$$a=sgn(C-t^2),$$
$$b=sgn^{\prime}(t^2+t-C),$$
and
$$c=sgn(t).$$
The solution is listed as
\begin{eqnarray}
\left\{\begin {array}{ll}
       m=c\{(2t-1)(1-a)+a[t(t+3)-(t+1)b-C]\}\\
       n=ca[C-t(t+1-b)]\\
       \end {array}\right.
\end{eqnarray}
\end{Corollary}

\section{Conclusions}\label{sec:concl}

In this paper, we indirectly introduce a general method to construct a symmetric binary pairing function mapping $D \times D$ onto $D$, where $D$ is a set of integers, illuminated by the construction of Cantor pairing function. Specially we present two pairing functions mapping $\mathbb {N^+} \times \mathbb {N^+}$ onto $\mathbb {N^+}$ and $\mathbb N \times \mathbb N$ onto $\mathbb N$ respectively, the former of which has been given a complete demonstration with the technique of set theory. These two pairing functions both have a property of symmetry shown as $F(m,n)=F(n,m)$ if denoted by $F(m,n)$. They are invertible and the corresponding inverse formulas are given.

\bibliographystyle{elsarticle-num}
\bibliography{bib}

\end{document}